\documentclass[12pt,twoside]{article}
\usepackage{amssymb,amsmath,amsthm,latexsym}
\usepackage{amsfonts}
\usepackage{enumerate}
\usepackage[demo]{graphicx}
\usepackage{relsize}
\usepackage{calc}
\usepackage{tikz-network}
\usepackage[caption=false]{subfig}
\usepackage{tikz}
\usetikzlibrary{decorations.markings}
\tikzstyle{vertex}=[circle, draw, inner sep=0pt, minimum size=6pt]
\usepackage{mathrsfs}
\usetikzlibrary{arrows.meta,positioning, quotes}
\newtheorem{thm}{Theorem}[section]

\newtheorem{cor}[thm] {Corollary}
\newtheorem{lem} [thm]{Lemma}

\theoremstyle{definition} % Definitions, examples, remarks, algorithms should be in roman type, not italic.

\newtheorem{defn}[thm]{Definition}
\numberwithin{equation}{section}

\newcommand\diag{\operatorname{diag}}

\def\ni{\noindent}
\tikzset{
  vertex/.style={circle, fill, inner sep=2pt},
  midarrow/.style={
    postaction={
      decorate,
      decoration={
        markings,
        mark=at position 0.5 with {\arrow{>}}
      }
    }
  }
}

\begin{document}
\label{'ubf'}
\setcounter{page}{1} %Put here the starting page number

\markboth {\hspace*{-9mm} \centerline{\footnotesize \sc
         % Put here the left page top label
    Distance Matrices for Conjugate Skew Gain Graphs}
                 }
                { \centerline {\footnotesize \sc
                   %put here the author's name
             Shahul Hameed K, Ramakrishnan K O and Biju K} \hspace*{-9mm}
               }
\begin{center}
{
       {\huge \textbf{Distance Matrices\\  for Conjugate Skew Gain Graphs
    % Put the title of the paper here
                               }
       }
\\

\medskip

Shahul Hameed K \footnote{\small  Department of
Mathematics, K M M Government\ Women's\ College, Kannur - 670 004,\ Kerala,  \ India.  E-mail: shahulhameed@kmmgwc.ac.in} 
Ramakrishnan K O \footnote{\small  Department of
Mathematics, K M M Government\ Women's\ College, Kannur - 670 004,\ Kerala,  \ India.  E-mail: ramkomaths@gmail.com} 
Biju K \footnote{\small Department of Mathematics, P R N S S College, Mattanur - 670\ 702,\ Kerala,\ India.\ Email: bijukaronnon@gmail.com}

}
\end{center}

\thispagestyle{empty}
\begin{abstract}
\ni  A conjugate skew gain graph is a skew gain graph with the labels (also called, the conjugate skew gains) from the field of complex numbes on the oriented edges such that they get conjugated when we reverse the orientation. In this paper we introduce distance matrices for conjugate skew gain graphs and characterize balanced conjugate skew gain graphs using these matrices. We provide explicit formulae for the distance spectra of certain conjugate skew gain graphs.
\\
---------------------------------------------------------------------------------------------\\
\end{abstract}
\textbf{Key Words:} Conjugate skew gain graphs, Distance matrix, Eigenvalues, Distance compatibility.\\
\textbf{Mathematics Subject Classification (2020):} \ 05C22, 05C50.
%\newpage
%\tableofcontents
%\setcounter{tocdepth}{3}
\section{Introduction}
In this paper, we deal with the ideas of distance and distance matrices for conjugate skew gain graphs which already exist in literature for the discrete structures such as graphs quite early (for example refer to \cite{fh1}) and more recently for signed graphs~\cite{shkdist} and complex unit gain graphs~\cite{kannan}. In~\cite{shkcsg}, we have extended the theory of balance to conjugate skew gain graphs characterizing it using their adjacency and $g$-Laplacian matrices.

Now we enlist some notations used in this paper. $\mathbb{C}^{\times}$ denotes the set of non-zero complex numbers. Also, $\Re(z)$ and $\Im(z)$ denote the real and imaginary parts of a complex number $z$ respectively. The reader may refer to \cite{fh,tz1,tz2, j1} in that order for the basic definitions and other details dealing with graphs, signed graphs, gain graphs and skew gain graphs. All the underlying graphs in this article are simple, finite and connected, unless otherwise mentioned. The symbol $\overrightarrow{E}$ stands for the collection of oriented edges such that for an edge $uv\in E$ of a graph $G=(V,E)$, we have two oriented edges $\overrightarrow{uv}$ and $\overrightarrow{vu}$ in $\overrightarrow{E}$. First of all, we begin with the definition of a conjugate skew gain graph.  
\begin{defn}(\cite{shkcsg}) \rm{Let $G=(V, E)$ be a graph with some prescribed orientation for the edges. A \emph{conjugate skew-gain graph} (or for brevity, a csg) $G^{\varphi}$ is such that the \emph{conjugate skew gain function} $\varphi:\overrightarrow{E}\rightarrow \mathbb{C}^\times$ satisfies $\varphi(\overrightarrow{vu})=\overline{\varphi(\overrightarrow{uv})}$.} 
\end{defn}  
The \emph{conjugate skew gain}, $\varphi(\overrightarrow{C})$, of an oriented cycle $\overrightarrow{C}:v_0v_1 \dots v_nv_0$ (i.e., it is a sequence of oriented edges $\overrightarrow{v_0v_1},\overrightarrow{v_1v_2},\cdots,\overrightarrow{v_nv_0}$) in a csg, is the product $\varphi(\overrightarrow{v_0v_1})\varphi(\overrightarrow{v_1v_2})\cdots\varphi(\overrightarrow{v_nv_0})$ of the conjugate skew gains of its edges and its conjugate, namely $\overline{\varphi(\overrightarrow{C})}$, is the conjugate skew gain $\varphi(\overleftarrow{C})$ of the reverse order oriented cycle $\overleftarrow{C}: v_0v_n \dots v_1v_0$. Similary, for an oriented path $\overrightarrow{P}: v_1v_2\cdots v_{k-1}v_k$ of length $k$, its conjugate skew gain is denoted by $\varphi(\overrightarrow{P})$ and $\varphi(\overleftarrow{P})$ denotes the conjugate skew gain of the oriented path in the recerse order given by $\overleftarrow{P}:v_kv_{k-1}\cdots v_2v_1$ so that $\varphi(\overleftarrow{P})=\overline{\varphi(\overrightarrow{P})}$. The notations $v\sim u$ and $v\sim e$ denote the adjacency of vertices $u$ and $v$ and that of the edge $e$ incident with the vertex $v$, respectively. We take the order of the underlying graph $G$ as $n$, quite often, of a csg $G^{\varphi}$, unless otherwise mentioned.
As we specifically deal with the balance in csgs, we provide the definition of the same as follows.
\begin{defn}(\cite{shkcsg}) \rm{A csg $G^{\varphi}$ is said to be balanced, if the conjugate skew gain $\varphi(\overrightarrow{C})=\prod_{e\in E(C)}\varphi(e)$ of every oriented cycle $\overrightarrow{C}$ in it satisfies $\varphi(\overrightarrow{C})=|\varphi(\overrightarrow{C})|$. i.e., if $\arg\big(\varphi(\overrightarrow{C})\big)=2p\pi$ for some integer $p$. }
\end{defn}
The balance theory for signed graphs, and complex unit gain graphs as well, coincides with the above one when we take $\varphi(C)=1$ for every cycle $C$. Note that the social balance theory, pioneered by Harary, Heider, Cartwright, et al., explores social networks using many tools including graph theory.  Historical instances like Heider's cognitive balance and Cartwright's sociological applications emphasize the significance of social balance theory. From a graph-theoretical perspective, social balance theory unveils mechanisms governing social relationships and group dynamics, aiding alliance formation, cooperation, and conflict resolution. 
\section{Distance matrices for csgs}\label{sec1}
By the notation $P_{(u,v)}$, we mean a path $P$ connecting the vertices $u$ and $v$. $\overrightarrow{P}_{(u,v)}$ is an oriented path from $u$ to $v$ and the collection of all such oriented shortest paths from $u$ to $v$ is written as $\mathcal{P}_{(u,v)}$. 
\begin{defn} Two vertices $u$ and $v$ in a csg $G^{\varphi}$ are said to be argument-wise distance compatible if for any two shortest oriented paths $\overrightarrow{P}_{(u,v)}$ and $\overrightarrow{Q}_{(u,v)}$ from $u$ to $v$, $\arg\Big(\varphi(\overrightarrow{P}_{(u,v)})/\varphi(\overrightarrow{Q}_{(u,v)})\Big)=\arg\big(\varphi(\overrightarrow{P}_{(u,v)})\big)-\arg\big(\varphi(\overrightarrow{Q}_{(u,v)})\big)=2p\pi$, for some integer $p$. Moreover, $G^{\varphi}$ itself is said to be argument-wise distance compatible if every pair of distinct vertices is argument-wise distance compatible. \\
Two vertices $u$ and $v$ are said to be modulus-wise distance compatible if for any two shortest oriented paths $\overrightarrow{P}_{(u,v)}$ and $\overrightarrow{Q}_{(u,v)}$ from $u$ to $v$, $|\varphi(\overrightarrow{P}_{(u,v)})|=|\varphi(\overrightarrow{Q}_{(u,v)})|$ or in other words $|\varphi(\overrightarrow{P}_{(u,v)})/\varphi(\overrightarrow{Q}_{(u,v)})|=1$. The csg $G^{\varphi}$ itself is said to be modulus-wise distance compatible if every pair of distinct vertices is modulus-wise distance compatible. \\
Two vertices $u$ and $v$ in a csg $G^{\varphi}$ are said to be distance compatible if they are both arugument-wise and modulus-wise distance compatible. $G^{\varphi}$ itself is said to be distance compatible if every pair of distinct vertices is distance compatible. 
\end{defn}
Using the lexicographical order in $\mathbb{C}$ (i.e., $a+ib<c+id$ if either $a<c$ or if $a=c$ then $b<d$), we define two quantities as follows.\\
$Q1: \varphi^{\max}_{(u,v)}=\max\{\varphi(\overrightarrow{P}_{(u,v)}):\overrightarrow{P}_{(u,v)}\in \mathcal{P}_{(u,v)} \}$\\
$Q2:\varphi^{\min}_{(u,v)}=\min\{\varphi(\overrightarrow{P}_{(u,v)}):\overrightarrow{P}_{(u,v)}\in \mathcal{P}_{(u,v)} \}$.\\
With respect to the two quantities defined above we define two complex valued distance measures for a connected csg $G^{\varphi}$ as:\\
$D1: d^{\max}(u,v)=\varphi^{\max}_{(u,v)}d(u,v)$ and \\
$D2: d^{\min}(u,v)=\varphi^{\min}_{(u,v)}d(u,v)$.\\
Further using the above, we define two types of distance matrices for a csg as follows:\\
$M1: D^{\max}(G^{\varphi})=\big(d^{\max}(u,v)\big)$ and \\
$M2: D^{\min}(G^{\varphi})=\big(d^{\min}(u,v)\big)$.

Note that in the case of a distance compatible csg $G^{\varphi}$, $d^{\max}(u,v)=d^{\min}(u,v)$ for all $u,v\in V(G)$ and so $D^{\max}(G^{\varphi})=D^{\min}(G^{\varphi})$. The common matrix in the case of a distance compatible csg $G^{\varphi}$ will be denoted by $D(G^{\varphi})$. Also for a distance compatible csg $G^{\varphi}$, the common value of $\varphi^{\max}_{(u,v)}$ and $\varphi^{\min}_{(u,v)}$ is denoted by $\varphi_{(u,v)}$.
\section{Some theorems on distance compatiblity}
In this section we provide some theorems on distance compatibility.
\begin{thm}\label{bimdist} Every balanced csg is argument-wise distance compatible.
\end{thm}
\begin{proof} Let $G^{\varphi}$ be a balanced csg. If possible assume on the contrary that there are two vertices $u$ and $v$ such that there exists shortest paths $P_{(u,v)}:uv_1v_2\cdots v_{k-1}v$ and $Q_{(u,v)}:uw_1w_2\cdots w_{k-1}v$ (of course, with some of the internal vertices are common for both the paths or all of the internal vertices are distinct) such that $\arg\big(\varphi(\overrightarrow{P}_{(u,v)})\big)-\arg\big(\varphi(\overrightarrow{Q}_{(u,v)})\big)\neq 2p\pi$ for any integer $p$. We deal with the following two cases.\\
Case 1: When all the internal vertices are distinct for the two paths $P_{(u,v)}$ and $Q_{(u,v)}$. Then $\overrightarrow{P}_{(u,v)}\bigcup \overleftarrow{Q}_{(u,v)}$ is the oriencted cycle $\overrightarrow{C}:uv_1v_2\cdots v_{k-1}vw_{k-1}w_{k-2}\cdots w_1u$. Thus, assuming that $\varphi(\overrightarrow{uv_1})=r_1e^{i\theta_1}, \varphi(\overrightarrow{v_1v_2})=r_2e^{i\theta_2},\cdots, \varphi(\overrightarrow{v_{k-1}v})=r_ke^{i\theta_k}$ and similarly, $\varphi(\overrightarrow{uw_1})=\rho_1e^{i\psi_1}, \varphi(\overrightarrow{w_1w_2})=\rho_2e^{i\psi_2},\cdots, \varphi(\overrightarrow{w_{k-1}v})=\rho_ke^{i\psi_k}$, we get $\varphi(\overrightarrow{C})=\big(\prod\limits_i r_i\rho_i \big ) e^{\sum\limits_i(\theta_i-\psi_i)}=|\varphi(\overrightarrow{C})|e^{\sum\limits_i(\theta_i-\psi_i)}$. Then applying the assumption that $G^{\varphi}$ is balanced, we get $e^{\sum\limits_i(\theta_i-\psi_i)}=1$ which implies $\sum\limits_i(\theta_i-\psi_i)=2q\pi$ for some integer $q$. This provides $\arg\big(\varphi(\overrightarrow{P}_{(u,v)})\big)-\arg\big(\varphi(\overrightarrow{Q}_{(u,v)})\big)=2q\pi$, a contradiction.\\
Case 2: When some of the internal vertices are common for both the paths, we can write $\overrightarrow{P}_{(u,v)}\bigcup \overleftarrow{Q}_{(u,v)}$ as the union of some oriented cycles $\overrightarrow{C}_j$ so that if $ \arg\big(\varphi(\overrightarrow{P}_{(u,v)})\big)-\arg\big(\varphi(\overrightarrow{Q}_{(u,v)})\big)\neq 2p\pi$, there is at least one cycle $C_k$ in that union of cycles for which $\varphi(\overrightarrow{C_k})\neq|\varphi(\overrightarrow{C_k})|$, making $G^{\varphi}$ unbalanced, leading to a contradiction in this case also.\\
Thus the balanced csgs are argument-wise distance compatible.
\end{proof}

The following figure illustrates that the converse of the above theorem need not be always true, as it is unbalanced though it is argument-wise distance compatible.
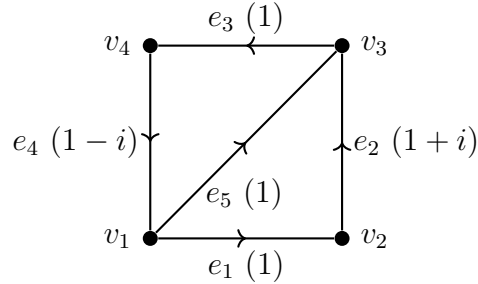
\begin{figure}[h!]
	\centering
	\begin{tikzpicture}[x=0.85cm, y=0.855cm]
		% Vertices
		\node[vertex, label=left:$v_1$] (u) at (0,0) {};
		\node[vertex, label=right:$v_2$] (v) at (3,0) {};
		\node[vertex, label=left:$v_4$] (w) at (0,3) {};
		\node[vertex, label=right:$v_3$] (x) at (3,3) {};
		
		% Edges with labels and arrow in the middle
		\path[midarrow, thick, draw] (u) to node[below]{$e_1\ (1)$} (v);
		\path[midarrow, thick, draw]  (v) to node[right]{$e_2\ (1+i)$} (x);
		\path[midarrow, thick, draw]  (x) to node[above]{$e_3\ (1)$} (w);
		\path[midarrow, thick, draw]  (w) to node[left]{$e_4\ (1-i)$} (u);
		\path[midarrow, thick, draw]  (u) to node[right,pos=0.075mm]{$e_5\ (1)$} (x);
	\end{tikzpicture}
	\caption{Counter example for the converse of Theorem~\ref{bimdist}; edge skew gains in the brackets}
	\label{fig4}
\end{figure}
Also, an argument-wise distance compatible and balanced csg need not be modulus-wise distance compatible as in Figure~\ref{fig5} given below.
\begin{figure}[h!]
	\centering
	\begin{tikzpicture}[x=0.85cm, y=0.855cm]
		% Vertices
		\node[vertex, label=left:$v_1$] (u) at (0,0) {};
		\node[vertex, label=right:$v_2$] (v) at (3,0) {};
		\node[vertex, label=left:$v_4$] (w) at (0,3) {};
		\node[vertex, label=right:$v_3$] (x) at (3,3) {};
		
		% Edges with labels and arrow in the middle
		\path[midarrow, thick, draw] (u) to node[below]{$e_1\ (e^{i\theta})$} (v);
		\path[midarrow, thick, draw]  (v) to node[right]{$e_2\ (2e^{-i\theta})$} (x);
		\path[midarrow, thick, draw]  (x) to node[above]{$e_3\ (3e^{-i\theta})$} (w);
		\path[midarrow, thick, draw]  (w) to node[left]{$e_4\ (4e^{i\theta})$} (u);
		%\path[midarrow, thick, draw]  (u) to node[right,pos=0.075mm]{$e_5\ (1)$} (x);
	\end{tikzpicture}
	\caption{Balanced and argument-wise distance compatible, but not modulus-wise}
	\label{fig5}
\end{figure}
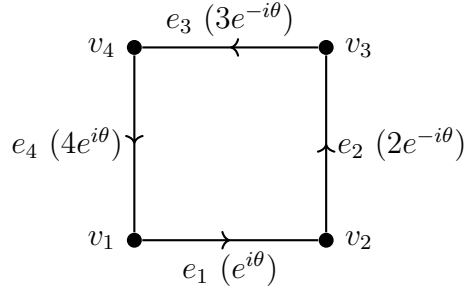
\begin{thm} A bipartite and connected csg $G^{\varphi}$ (i.e., the underlying graph $G$ is bipartite and connected) is balanced if and only if it is argument-wise distance compatible.
\end{thm}
\begin{proof} Using Theorem~\ref{bimdist}, one way implication follows easily. For the converse, assume on the contrary that $G^{\varphi}$, the bipartite connected csg, is not balanced even though it is argument-wise distance compatible. Then there is a shortest oriented cycle $ \overrightarrow{C}:uv_1v_2\cdots v_{k-1}v w_1 w_2\cdots w_{k-1}u $ of even length say $2k$, such that $\varphi(\overrightarrow{C})\neq |\varphi(\overrightarrow{C})|$. This means that $\arg\big(\varphi(\overrightarrow{C})\big)\neq 2p\pi$ for all integers $p$. But the two oriented paths $\overrightarrow{P}_{(u,v)}:uv_1v_2\cdots v_{k-1}v$ and $\overrightarrow{Q}_{(u,v)}=uw_{k-1}w_{k-2}\cdots w_1v$ that can be constructed from the above cycle give
\begin{align*} \arg\big(\varphi(\overrightarrow{P}_{(u,v)})\big)-\arg\big(\varphi(\overrightarrow{Q}_{(u,v)})\big)&= \arg\big(\varphi(\overrightarrow{P}_{(u,v)})\big)+\arg\big(\varphi(\overleftarrow{Q}_{(u,v)})\big)\\
&=\arg\big(\varphi(\overrightarrow{C})\big)\\
&\neq 2p\pi
\end{align*} for all integers $p$ and it leads to a contradiction. Thus $G^{\varphi}$ must be balanced.
\end{proof}
\begin{thm} A connected csg $G^{\varphi}$ is distance compatible if and only if every block of it is distance compatible.
\end{thm}
\begin{proof} The technique of the proof is almost same as that in \cite{shkdist} dealing with a similar result in the case of signed graphs and that for complex unit gain graphs given in \cite{kannan} . So we adapt the same here also to suit it for csgs.
Assuming first that  $G^{\varphi}$ is distance compatible, it is trivial that every block of it must be so, otherwise it will violate the distance compatibility of the entire csg itself.	\\
Conversely let $B_1, B_2,\cdots, B_p$ be the blocks of $G^{\varphi}$. Suppose that each of these blocks is distance compatible. Let $u,v\in V (G)$. If $u$ and $v$ are in the same block then they are clearly distance compatible. On the other hand, if $u$ and $v$ are in different blocks, without loss of generality,
suppose $u$ lies in $B_1$ and $v$ in $B_2$. Then any shortest path $P_{(u,v)}$ in $G$ must pass through the cut vertices $v_i$ and $v_j$ where $v_i$ and $v_j$ are in $B_1$ and $B_2$, respectively (of course, $v_i$ may be same as $v_j$). Any shortest path $P_{(u,v)}$ can be decomposed into $P_{(u,v_i)}\cup P_{(v_i,v_j)}\cup P_{(v_j,v)}$.
Since $B_1$ is distance compatible, so any shortest path has unique gain. As the vertices $v_i$ and $v_j$ are connected by a unique path, so $\varphi_{(v_i , v_ j)}$ is unique. Similarly we can claim the result for other cases. Therefore, any shortest path from $u$ to $v$ has same conjugate skew gain. Hence in this case also $u$ and $v$ are distance compatible. Thus, $G^{\varphi}$ is distance compatible.
\end{proof}
\section{Characterization of balance using distance matrices}
In this section we provide a characterization of connected balanced csgs in terms of the distance matrices defined in Section~\ref{sec1}. First, we shall deal with the impact of switching on the compatibility in a connected csg. We denote by $\mathbb{T}$, the set of all complex numbers with modulus one. i.e., $\mathbb{T}=\{z\in\mathbb{C}:|z|=1\}$. By a switching we mean the following. A function $\zeta:V \rightarrow \mathbb{T}\subset \mathbb{C}^\times $ is called a \textit{switching function} for a csg $G^{\varphi}$, if it results in a csg $G^{\varphi^\zeta}$ where $\varphi^{\zeta}$ satisfies $ \varphi^{\zeta}(\overrightarrow{uv})=\overline{\zeta(u)}\varphi(\overrightarrow{uv})\zeta(v)$.
\begin{thm}\label{switch0} In a connected csg $G^{\varphi}$, a switching does not change the pairs of distance compatible vertices.
\end{thm}
\begin{proof} Let $u$ and $v$ be a pair of distance compatible vertices in the connected csg $G^{\varphi}$ and let $\zeta:V\rightarrow \mathbb{T}$ be any switching function. Then for any oriented path $\overrightarrow{P}_{(u,v)}$, we have $\varphi^\zeta(\overrightarrow{P}_{(u,v)})=\overline{\zeta(u)}\varphi(\overrightarrow{P}_{(u,v)})\zeta(v)$. Thus, if $|\varphi(\overrightarrow{P}_{(u,v)})|=|\varphi(\overrightarrow{Q}_{(u,v)})|$ and if $\arg\big(\varphi(\overrightarrow{P}_{(u,v)})\big)-\arg\big(\varphi(\overrightarrow{Q}_{(u,v)})\big)= 2p\pi$ for some integer $p$ for any two oriented paths $\overrightarrow{P}_{(u,v)}$ and $\overrightarrow{Q}_{(u,v)}$, then an easy calculation shows that $|\varphi^\zeta(\overrightarrow{P}_{(u,v)})|=|\varphi^\zeta(\overrightarrow{Q}_{(u,v)})|$ and $\arg\big(\varphi^\zeta(\overrightarrow{P}_{(u,v)})\big)-\arg\big(\varphi^\zeta(\overrightarrow{Q}_{(u,v)})\big)= 2p\pi$ for the same integer $p$. Hence switching preserves all distance compatible pairs of vertices. In fact, in the above calculation one has to use $|\zeta(u)|=1=|\zeta(v)|$ and for the arguments they get cancelled while the subtraction is being done.
\end{proof}
 \begin{thm}\label{switch2}
	If $G^{\varphi}$ is switched to $G^{{\varphi}^\zeta}$ and if $D^{\max}(G^{\varphi})=D^{\min}(G^{\varphi})=D(G^{\varphi})$ then $D^{\max}(G^{{\varphi}^\zeta})=D^{\min}(G^{{\varphi}^\zeta})=D(G^{{\varphi}^\zeta})$ and $D(G^{\varphi})$ is similar to $D(G^{{\varphi}^\zeta})$.
\end{thm}
\begin{proof} Let $V=\{v_1,v_2,\cdots, v_n\}$ and $\zeta:V\rightarrow \mathbb{T}$ be the switching function such that $\zeta(v_j)=e^{i\theta_j}$ for $j=1, 2,\cdots, n$. Now the equalities $D^{\max}(G^{{\varphi}^\zeta})=D^{\min}(G^{{\varphi}^\zeta})=D(G^{{\varphi}^\zeta})$ follow easily from the above Theorem~\ref{switch0}. To prove the part dealing with the similarity, taking the diagonal matrix $S=\diag(e^{i\theta_j})_{n\times n}$, we get $D(G^{{\varphi}^\zeta})=SD(G^{\varphi})S^*$ where $S^*$ is the conjugate transpose of $S$ and in this case, $S^*=S^{-1}$. This completes the proof.
\end{proof}
Before we characterize balanced csgs using distance matrices, we need, first of all, some results from~\cite{shkcsg}, the concept of complete csgs associated with a given distance compatible csg $G^{\varphi}$ and the characterization of balance in these associated csgs.  It may also be noted that the distance matrices of a distance compatible csg $G^{\varphi}$ and its counter part $G^{|\varphi|}$ can be thought of as the adjacency matrices of the corresponding associate csgs $K_n^{\varphi'}$ and $K_n^{|\varphi'|}$ respectively where $\varphi'$ is the conjugate skew gain function defined by $\varphi'(\overrightarrow{uv})=\varphi_{(u,v)}d(u,v)$ built on the underlying complete graph $K_n$ and $n$ is the number of vertices or the order of $G$. Note that $G^{\varphi}$ is a spanning subgraph of $K_n^{\varphi'}$. Recall also that the adjacency matrix $A(G^{\varphi})=(a_{ij})$ of a csg $G^{\varphi}$ is defined~\cite{shkcsg} as:\\
$a_{ij} =
\left\{
\begin{array}{ll}
\varphi(\overrightarrow{v_iv_j})  & \mbox{if } v_i\sim v_j \\
0 & \mbox{otherwise }
\end{array}
\right.$ \\ such that whenever $a_{ij}\neq 0,  \ 
a_{ji}=\varphi(\overrightarrow{v_jv_i}) =\overline{a}_{ij}$.
\begin{thm}(\cite{shkcsg})\label{switchshk}
 A csg $G^{{\varphi}}$ is balanced if and only if it is switching equivalent to $G^{{|\varphi|}}$.
 \end{thm}
\begin{thm}(\cite{shkcsg})\label{switchshk1}
 A csg $G^{{\varphi}}$ is balanced if and only if $A\big(G^{{|\varphi|}}\big)$ and $A\big(G^{{\varphi}}\big)$ have the same spectrum.
 \end{thm}
Now we characterize balance in a connected csg using its distance matrix. In the proof of the theorem we require the formula found in~\cite{shkcsg} for the coefficients in the characteristic polynomial of the adjacency matrix of a csg.
In Equation~\eqref{eq2} that follows, $\mathfrak{L}_{i}$ denotes the set of all elementary subgraphs of order $i$ where by an elementary subgraph we mean that subgraph of $G$ whose components either an edge $K_2$ or a cycle $C$. $C(L)$ denotes the number of cycles in an elementary subgraph $L$ and $K(L)$ denotes the number of components in that $L$.
\begin{lem}(\cite{shkcsg})\label{char} If $G^{\varphi}$ is a csg where $G=(V,\ E)$ is a graph of order $n$ and if $\det\big(xI-A(G^{\varphi})\big)=\sum\limits_{i=0}^{n}a_i(G^{\varphi})x^{n-i}$ is the characteritic polynomial, then
\begin{equation}\label{eq2} a_i(G^{\varphi})=\sum\limits_{L\in\mathfrak{L}_{i}}(-1)^{K(L)}2^{C(L)}\Big(\prod\limits_{K_2\in L}\prod\limits_{e\in K_2}|\varphi(e)|^2\Big)\Big(\prod\limits_{C\in L}\Re(\varphi(C))\Big)
\end{equation}
\end{lem}
\begin{thm}\label{reduced1} A modulus-wise distance compatible csg $G^{\varphi}$ is balanced if and only if the associated csg $K_n^{\varphi'}$ is balanced.
\end{thm}
\begin{proof} Assume first that $K_n^{\varphi'}$ is balanced. This immediately gives the fact that $G^{\varphi}$ is balanced as it is a subgraph of  $K_n^{\varphi'}$. Conversely assume that the modulus-wise distance compatible csg $G^{\varphi}$ is balanced. Then the switching function $\zeta$ which switches $G^{\varphi}$ to $G^{|\varphi|}$ will also definitely switch $K_n^{\varphi'}$ to $K_n^{|\varphi'|}$. Hence by Theorem~\ref{switchshk1}, the proof is complete.
\end{proof}
\begin{thm}\label{char}
A connected csg $G^{\varphi}$ is balanced and modulus-wise distance compatible if and only if $D(G^{\varphi})$ exists and it is cospectral with $D(G^{|\varphi|})$.
\end{thm}
\begin{proof} Assume first that $G^{\varphi}$ is balanced and modulus-wise distance compatible. Thus by Theorem~\ref{bimdist}, it is distance compatible and as such the matrix $D(G^{\varphi})$ exists. As $G^{\varphi}$ is balanced, by Theorem~\ref{switchshk}, it is switching equivalent to $G^{|\varphi|}$. This  means $D(G^{\varphi})$ is similar to $D(G^{|\varphi|})$ and hence they are cospectral. In fact, if $S=\diag(\zeta(v_i))$, the diagonal matrix corresponding to the switching function $\zeta$, then $D(G^{\varphi})=S^{-1}D(G^{|\varphi|})S$.\\
Conversely assume that $D(G^{\varphi})$ exists and it is cospectral with $D(G^{|\varphi|})$. As $D(G^{\varphi})$ exists, $G^{\varphi}$ must be distance compatible and hence modulus-wise distance compatible. So it remains to prove that $G^{\varphi}$ is balanced. Due to Lemma~\ref{reduced1}, it is enough to prove that $K_n^{\varphi'}$ is balanced. Assume on the contrary, let there be unbalanced cycle or cycles of shortest length, say $i\ge 3$, in $K_n^{\varphi'}$ which form part of some elementary subgraph of order $i$. Denoting the collection of such cycles as $\mathfrak{C}_i\subseteq \mathfrak{L}_i $ and noting the fact that other terms for $a_i(K_n^{|\varphi'|})$ and $a_i(K_n^{\varphi'})$ being equal, from Equation~\eqref{eq2},
\begin{align*} a_i(K_n^{|\varphi'|})-a_i(K_n^{\varphi'})&=(-2)\sum\limits_{C\in\mathfrak{C}_{i}}\Big(|\varphi'(C)|-\Re(\varphi'(C))\Big)\\
&\neq 0
\end{align*}
This contradiction to cospectrality of the adjacent matrices of $K_n^{\varphi'}$ and $K_n^{|\varphi'|}$, by Theorem~\ref{switchshk1}, completes the proof.
\end{proof}
Here is an example of a balanced csg illustrating Theorem~\ref{char}. 
\begin{figure}[h!]
	\centering
	\begin{tikzpicture}[x=0.85cm, y=0.855cm]
		% Vertices
		\node[vertex, label=left:$v_1$] (u) at (0,0) {};
		\node[vertex, label=right:$v_2$] (v) at (4,0) {};
		\node[vertex, label=left:$v_4$] (w) at (0,4) {};
		\node[vertex, label=right:$v_3$] (x) at (4,4) {};
		
		% Edges with labels and arrow in the middle
		\path[midarrow, thick, draw] (u) to node[below]{$e_1\ (1+i)$} (v);
		\path[midarrow, thick, draw]  (v) to node[right]{$e_2\ (i)$} (x);
		\path[midarrow, thick, draw]  (x) to node[above]{$e_3\ (1-i)$} (w);
		\path[midarrow, thick, draw]  (w) to node[left]{$e_4\ (-i)$} (u);
		\path[midarrow, thick, draw]  (u) to node[right,pos=0.075mm]{$e_5\ (-1+i)$} (x);
	\end{tikzpicture}
	\caption{Example to illustrate Theorem~\ref{char}; edge skew gains in the brackets}
	\label{fig4}
\end{figure}
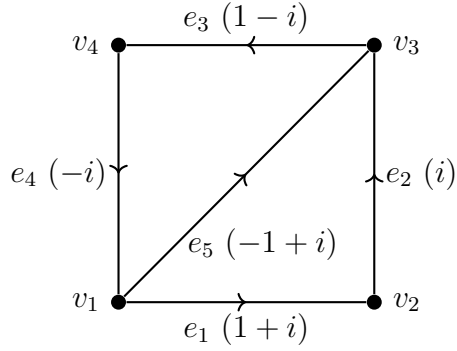
\\ In fact, for the csg in Figure~\ref{fig4}, the distance matrices $D(G^{\varphi})$ and $D(G^{|\varphi|})$ are as follows. 
$$D(G^{\varphi})=\begin{pmatrix}0 & 1+i & -1+i & i\\   1-i & 0 & i & 2(1+i)\\-(1+i) &  -i & 0 & 1-i\\-i & 2(1-i) & 1+i & 0\end{pmatrix}$$
$$D(G^{|\varphi|})=\begin{pmatrix}0 & \sqrt{2} & \sqrt{2} & 1\\\sqrt{2} & 0 & 1 & 2\sqrt{2}\\ \sqrt{2} &  1 & 0 & \sqrt{2}\\1 & 2\sqrt{2} & \sqrt{2}& 0 \end{pmatrix}$$

Their common characteristic polynomial is $\lambda^4-16\lambda^2-24\lambda-7$ and the common spectra is
 $\begin{pmatrix}\lambda_1&\lambda_2 &\lambda_3& \lambda_4 \\ 1 & 1 & 1& 1 \end{pmatrix}$ where $\lambda_1=-\lambda_2=\big(-\sqrt{14-2^{7/2}}-3\sqrt{2}\big)/2$ and $\lambda_3=-\lambda_4=\big(-\sqrt{14-2^{7/2}}+3\sqrt{2}\big)/2$.
\section{Distance spectra of certain csgs}
In this section, we compute the eigenvalues of distance matrices of certain csgs which are distance compatible. First we consider a csg built on the odd cyle $C_n$ where $n=2p+1$ and the conjugate skew gains are chosen from the circle of radius $k$ in the complex plane with center at the origin. i.e., $|\varphi(\overrightarrow{e})|=k$ for all the edges of the cycle. 
First we have the following lemma which we use to find the sum in the computation of distance spectrum.
\begin{lem}\label{sum}
$$\sum\limits_{r=1}^{p}rk^r\cos(r\theta)=\dfrac{f(\theta)}{g(\theta)}$$\\ 
where 
\begin{align*}
    \begin{split}
        f(\theta) &= pk^{p+2}\cos(p+2)\theta - k^{p+1}\big(p(2k^2+1)+1\big)\cos(p+1)\theta \\
        &\quad + k^{p+2}\big(p(k^2+2)+2\big)\cos p\theta  -(p+1)k^{p+3}\cos(p-1)\theta + k(k^2+1)\cos\theta - 2k^2\\
    \end{split} 
\end{align*}
 and $g(\theta)=(1-2k\cos\theta+k^2)^2$
\end{lem}
\begin{proof} Let $S=\sum\limits_{r=1}^{p}rk^r\cos(r\theta)$. Then 
\begin{align*} S&=\frac{1}{2}\sum\limits_{r=1}^{p}\Big[rk^r\big(e^{ir\theta}+e^{-ir\theta} \big)\Big]
                       =\frac{1}{2}[S_1+\overline{S_1}]
                       =\Re(S_1)
\end{align*}
where, by choosing $z=ke^{i\theta},\ S_1=\sum\limits_{r=1}^{p}rz^r$. Thus using arithmetico geometric series,
\begin{align*} S_1&=\dfrac{z(1-z^p)}{(1-z)^2}-\dfrac{p\ z^{p+1}}{1-z}
                          =\dfrac{z-(p+1)z^{p+1}+pz^{p+2}}{(1-z)^2}
\end{align*}
Multiplying by conjugate of the denominator and simplifying further using the fact that $|z|=k$,
\begin{align*} S_1&=\dfrac{\Big(z-(p+1)z^{p+1}+pz^{p+2}\Big)(1-\overline{z})^2}{(1-2k\cos\theta+k^2)^2}\\
                          &=\dfrac{\Big(z-(p+1)z^{p+1}+pz^{p+2}\Big)(1-2\overline{z}+(\overline{z})^2)}{(1-2k\cos\theta+k^2)^2}\\
                          &=\dfrac{\Big(z-2k^2+k^2\overline{z}-((p+1)+2pk^2)z^{p+1}+pz^{p+2}+(2(p+1)k^2+pk^4)z^{p}-(p+1)k^4z^{p-1}\Big)}{(1-2k\cos\theta+k^2)^2}
\end{align*}
This simplifies to the required sum $\Re(S_1)=\dfrac{f(\theta)}{g(\theta)}$ as desired.
\end{proof}
\begin{thm}
	For an odd unbalanced cycle $C_n^{\varphi}$ where $n=2p+1$ with $\theta=\arg(\varphi(C_n^\varphi))$  the distance spectrum is given by: 
	$$\lambda_j=\dfrac{2f(\theta_j)}{g(\theta_j)}$$ where $\theta_j=\frac{{(2\pi j+\theta)}}{n}, j=0,1,\cdots, n-1$ and the functions $f$ and $g$ are as in Lemma~\ref{sum}.
\end{thm}
\begin{proof}
Let the given cycle $C_n$ be $v_1e_1v_2e_2v_3\cdots v_{n}e_nv_1$. By properly switching, without loss of generality, the chosen csg $C_n^\varphi$ can be thought of as the one with $\varphi(e_n)=ke^{i\theta}$ where $\theta=\arg(\varphi(C_n^\varphi))$ and $\varphi(e_j)=k$ for the remaining edges. Using Theorem~\ref{switch0}, it is enough to compute the distance spectra of this switched cycle. The distance matrix $D(C_n^\varphi)$ can be written as\\
$\begin{pmatrix}
	0 & k & 2k^2 & \cdots pk^p & pk^pe^{-i\theta} & (p-1)k^{p-1}e^{-i\theta} \cdots & 2k^2e^{-i\theta} & ke^{-i\theta}\\
	k & 0 & k & 2k^2 \cdots \ \ \ \  pk^p & pk^pe^{-i\theta} & \cdots & \cdots & 2k^2e^{-i\theta}\\
	\cdots & \cdots & \cdots & \cdots \cdots &\cdots& \cdots & \cdots &\cdots\\
	\cdots & \cdots & \cdots & \cdots \cdots &\cdots& \cdots & \cdots &\cdots\\	 
	ke^{i\theta} & 2k^2e^{i\theta} & \cdots & \cdots   &\cdots & \cdots& k& 0
\end{pmatrix}
$

The matrix being a real symmetric one, its eigenvalues are real.	
To find the eigenvalues we first choose an eigenvector corresponding to the eigenvalue $\lambda$ as $X=[\rho,\rho^2,\ldots,\rho^{n}]^T$, where $\rho$ is a complex number to be found out. The eigenalue-eigenvector equation $D(C_n^\varphi)X=\lambda X$ gives $\lambda=\sum\limits_{r=1}^{p}r\rho^{r}+e^{-i\theta}\sum\limits_{r=1}^{p}r\rho^{n-r}$. Putting this value in any of the remaining equation, corresponding to the system of equations for $D(C_n^\varphi)X=\lambda X$, will give us $\rho^n=e^{i\theta}$ so that $\rho=e^{\frac{{(2\pi j+\theta)}i}{n}}$ for $j=0,1,\cdots, n-1$. Thus the eigenvalues are 
\begin{align*}
\lambda_{j}&=\sum\limits_{r=1}^{p}rk^r\rho^{r}+ e^{-i\theta}\sum\limits_{r=1}^{p}rk^r\rho^{n-r}\\
                &=\sum\limits_{r=1}^{p}rk^r\big(\rho^{r}+\rho^{-r}\big)\\
                &=2\sum\limits_{r=1}^{p}rk^r\cos{(r\theta_{j})}
\end{align*}
where $\theta_j=(2\pi j+\theta)/n$. From this using Lemma~\ref{sum}, the eigenvalues are: 
$\lambda_j=\dfrac{2f(\theta_j)}{g(\theta_j)}$ where $\theta_j=(2\pi j+\theta)/n, j=0,1,\cdots, n-1$ with the functions $f$ and $g$ as defined above.
\end{proof}
As already mentioned, a complex unit gain graph is a csg with the conjugate skew gains (called simply as gains in their case) selected from the unit circle $\mathbb{T}=\{z\in\mathbb{C}:|z|=1\}$. Thus we have the following corollary that gives the distance spectrum of complex unit gain graphs obtained by choosing $k=1$ for which the proof is omitted as it is easy to derive from the above theorem. 
\begin{cor} For an odd unbalanced complex unit gain cycle $C_n^{\varphi}$ where $n=2p+1$ with $\theta=\arg(\varphi(C_n^\varphi))$, the distance spectrum is given by: 
	$$\lambda_j=\dfrac{n\sin\big({n\theta_j/2}\big)-\frac{n-1}{2}\sin\big({(n+2)\theta_j/2}\big)-\frac{n+1}{2}\sin\big({(n-2)\theta_j/2}\big)-2\sin\big({\theta_j/2}\big)}{4\sin^3\big({\theta_j/2}\big)}$$ where $\theta_j=\frac{{(2\pi j+\theta)}}{n}, j=0,1,\cdots, n-1$.
\end{cor}
\section*{Acknowledgement} The second and the third authors are research scholars at Kannur University, Kerala, India and they are thankful to the Research Directorate of Kannur University for providing access to the library and lab facilities there.

\end{document}